\documentclass[draft]{article}
\usepackage[dvips]{epsfig}
\input{epsf}
\usepackage{amsmath,amssymb,amsfonts}
\usepackage{amscd}
\usepackage[dvips,matrix,arrow]{xy}

\newtheorem{theorem}{Theorem}%[section]
\newtheorem{lemma}[theorem]{Lemma}
\newtheorem{proposition}[theorem]{Proposition}
\newtheorem{corollary}[theorem]{Corollary}

\newtheorem{remark}[theorem]{Remark}

\newenvironment{proof}[1][Proof:]{\begin{trivlist}
\item[\hskip \labelsep {\bfseries #1}]}{\end{trivlist}}

\newcommand{\qed}{\nobreak \ifvmode \relax \else
      \ifdim\lastskip<1.5em \hskip-\lastskip
      \hskip1.5em plus0em minus0.5em \fi \nobreak
      \vrule height0.75em width0.5em depth0.25em\fi}

 0 3 \font\ccc =msbm10

\begin{document}

\title{On certain permutation representations of the braid group}

\author{Valentin Vankov Iliev}

\maketitle

\section*{Abstract}

In this paper we present a structural theorem,
concerning certain homomorphic images of Artin braid group on $n$
strands in finite symmetric groups. It is shown that each one of 
these permutation groups is an extension of the symmetric group 
$S_n$ by an appropriate abelian group, and in "half" of 
the cases this extension splits.  
\vskip 10pt

\noindent {\bf Mathematics Subject Classification 2010}: 20F36,
20E22

\noindent {\bf Key words}: Artin braid group, permutation
representation, split extension

\noindent {\bf Mailing address}: Section of Algebra, Institute of
Mathematics and Informatics, Bulgarian Academy of Sciences, 1113
Sofia, Bulgaria

\section{Introducton}

\label{I}

In \cite{[1]}, E.~Artin studies the permutation representations of
his braid group $B_n$ in the symmetric group $S_n$ as a first step
toward the determination of all automorphisms of $B_n$. He
considers the last problem to be "quite difficult" and notes:
"Before one can attack it, one has obviously first to get the
permutations aut of the way."

The aim of this paper is to present the proof of a structure
theorem concerning certain permutation representations of $B_n$.
The corresponding finite permutation groups are defined as
subgroups of a wreath product, and this allows their description as an 
(in "half" of the cases, split) extension of 
the symmetric group $S_n$ by an abelian group.

The text is organized as follows. In the beginning of
Section~\ref{III.2} we remind some terminology and simple
statements, concerning wreath products of the type $W\wr S_n$,
where $W\leq S_d$ is a permutation group, and use a special
endomorphism $\omega$ of the symmetric group $S_\infty$ in order
to simplify the notation. The permutations $\sigma$ from the
non-trivial coset of the base group $W^{\left(2\right)}$ of $W\wr
S_2$ play a crucial role in this paper. If we consider such a
$\sigma$, and its translate $\omega^d(\sigma)$,
Theorem~\ref{III.2.12} establishes four conditions that are
equivalent to the fact that this pair of permutations is
braid-like. As a by-product we obtain four necessary and sufficient
conditions for two permutations in $S_d$ to commute in terms of 
(the cycle structure of) their common product. In particular, 
one of these conditions gives explicit expressions for the two 
commuting permutations, see Remark~\ref{III.2.14}. 

Section~\ref{III.3} is devoted to the description of certain
$n$-braid-like groups $B_n(\sigma)$ (that is, homomorphic images
of $B_n$) which are subgroups of some finite symmetric group.
These groups are generated by the braid-like couples $\sigma$,
$\omega^d(\sigma)$, and several of their $\omega^d$-translates.
The permutation $\sigma$ depends on a fixed permutation $\tau\in
S_d$, and on a permutation $u$ of the set 
$\mathcal{C}\left(\tau\right)$ of cycles of $\tau$, 
see~\ref{III.2}. In
Proposition~\ref{III.3.30} we show that, in general, the group
$B_n(\sigma)\leq S_d\wr S_n$ is intransitive, find its sets of
transitivity $(Y_o)_{o\in\mathcal{C}\left(\tau\right)}$, as well
as the restrictions $B_n(\sigma^{\left(o\right)})$ of
$B_n(\sigma)$ on $Y_o$, which are groups of the same type.
Moreover, $B_n(\sigma)$ is a sub-direct product of
$B_n(\sigma^{\left(o\right)})$'s. In particular, the group
$B_n(\sigma)$ is transitive if and only if the permutation $u$ is
a long cycle. Let $q$ be 
the order of the permutation $\tau$. Theorem~\ref{III.3.4}
is the central result of the paper, where we prove that each
$n$-braid-like group $B_n(\sigma)$ is an extension of 
the symmetric group $S_n$ by an abelian group. If we assume,
in addition, that $q$ is odd, then this extension splits, and  
we can use the method of "little groups" of Wigner and Mackey (see
\cite[Proposition 25]{[33]}) for finding, in principle, of all
finite-dimensional irreducible representations of $B_n(\sigma)$.
Thus, we can obtain a series of finite-dimensional irreducible
representations of Artin braid group $B_n$.

\section{Braid-like pairs of permutations}

\label{III.2}

First, we introduce some notation. We define the symmetric group
$S_\infty$ as the group of all permutations of the set of positive
integers, which fix all but finitely many elements. The symmetric
group $S_d$ is identified with the subgroup of $S_\infty$,
consisting  of  all permutations fixing any $k>d$. The conjugation 
by a permutation $\zeta$ in $S_\infty$ is denoted by $c_\zeta$. 

From now on we assume that $d$ and $n$ are integers, $d\geq 2$,
$n\geq 3$. Given a permutation $\zeta\in S_d$, we denote by
$\varrho\left(\zeta\right)= (1^{c_1\left(\zeta\right)},
2^{c_2\left(\zeta\right)},\hdots, d^{c_d\left(\zeta\right)})$ its
cycle type (here $c_i\left(\zeta\right)$ is the number of cycles
of $\zeta$ of length $i$). For any partition $\lambda$ of $d$,
$\lambda=(1^{m_1},2^{m_2},\hdots, d^{m_d})$, we set
$z_\lambda=1^{m_1}m_1!2^{m_2}m_2!\hdots d^{m_d}m_d!$.

We denote by $\omega$ the injective endomorphism of $S_\infty$,
defined via the rule
\[
(\omega(\sigma))(k)=\sigma(k-1)+1,\hbox{\ }k\geq
2,\hbox{\ }(\omega(\sigma))(1)=1.
\] As usual, we identify the wreath product $S_d\wr S_n$ with the
image of its natural faithful permutation representation, see
\cite[4.1.18]{[6]}, and for each subgroup $W\leq S_d$ we identify
the wreath product $W\wr S_n$ with its image via the above
inclusion.

Let $\theta_s\in S_{nd}$, be the involutions
\[
\theta_s=\left(
\begin{array}{ccccccccccccccc}
(s-1)d+1 &  \cdots & sd & sd+1 &\cdots & (s+1)d\\
sd+1 & \cdots & (s+1)d & (s-1)d+1 &\cdots & sd
\end{array}
\right)\in S_{nd},
\]
$s=1,\ldots, n-1$. We set $\Sigma_n=
\langle\theta_1,\theta_2,\ldots,\theta_{n-1}\rangle\leq S_{nd}$.
Then the direct product $W^{\left(n\right)}=W\omega^d(W)\ldots
\omega^{\left(n-1\right)d}(W)$ is a normal subgroup of of the
wreath product $W\wr S_n$ (its base group), $\Sigma_n$ is a
complement of $W^{\left(n\right)}$, and we have that $W\wr S_n$ is
the semidirect product of $\Sigma_n$ by $W^{\left(n\right)}$:
$W\wr S_n=W^{\left(n\right)}\cdot \Sigma_n$. The isomorphism
$\Sigma_n\simeq S_n$
maps the involution $\theta_s$ onto the transposition $(s,s+1)$,
$s=1,\ldots n-1$.

In particular, for any $W\leq S_d$ the wreath product
$W\wr S_2$ is the semidirect product of
$\Sigma_2=\langle\theta\rangle$ by its base group
$W^{\left(2\right)}$, where $\theta=\theta_1$. The left coset
$\theta W^{\left(2\right)}$ of $W^{\left(2\right)}$ in
$W\wr S_2$ consists of permutations of $[1,2d]$
that map $[1,d]$ onto $[d+1,2d]$.

Let $\zeta$ be a permutation of a finite set. By
$\mathcal{C}_m(\zeta)$ we denote the set of all cycles of length
$m$ of $\zeta$, and set $\mathcal{C}(\zeta)=\cup_{m\geq
1}\mathcal{C}_m(\zeta)$. Let us fix a permutation $\tau\in W$. We
denote by $M_{W,\tau}$ the set of all permutations
$\sigma\in\theta W^{\left(2\right)}$ that satisfy the equation
\begin{equation}
\sigma^2=\tau \omega^d(\tau),\label{III.2.7}
\end{equation}
and let $M_W=\cup_{\tau\in S_d}M_{W,\tau}$. We set
$M_{d,\tau}=M_{S_d,\tau}$, and $M_d=M_{S_d}$. Let us denote by
$N_{d,\tau}$ the set of all permutations $\sigma\in S_d\wr S_2$,
constructed in the following way:

(A) fix a permutation $u$ of the set $\mathcal{C}(\tau)$,
that maps the subset $\mathcal{C}_m(\tau)$ onto itself for any
$m=1,\ldots,d$;

(B) choose an $m=1,\ldots,d$;

(C) for any $\alpha\in \mathcal{C}_m(\tau)$ choose an initial
element $i_1$ of $\alpha$,  an initial element $j_1$ of
$u_m(\alpha)$, and write $\alpha$ and $u(\alpha)$, using the cycle
notation:
\[
\alpha=(i_1,i_2,\ldots, i_m),
\]
\[
u(\alpha)=(j_1,j_2,\ldots, j_m);
\]

(D) shuffle the cycle notations of $\alpha$ and
$\omega^d(u(\alpha))$ from (C), and get a cycle of length $2m$:
\begin{equation}
(i_1, j_1+d,i_2, j_2+d,\ldots, i_m, j_m+d);\label{III.2.10}
\end{equation}

(E) multiply the cycles~(\ref{III.2.10}) for all
$\alpha\in\mathcal{C}(\tau)$, and
denote the permutation thus obtained by $\sigma$.

If in the above construction sequence (A), (B), (C), (D), (E), we
replace the steps (D), (E) with the steps (D$'$), (E$'$), (F$'$)
below, the result is a binary relation $R_{d,\tau}$ on $S_d$. The
set $R_{d,\tau}$ consists of all ordered pairs
$(\varsigma_1,\varsigma_2)\in S_d\times S_d$, obtained by the
construction sequence (A), (B), (C), (D$'$), (E$'$), (F$'$), where

(D$'$) denote by $p_\alpha$ the bijection
\[
supp(\alpha)\to supp(u(\alpha)),\hbox{\ } i_1\mapsto j_1,\ldots,
i_m\mapsto j_m,
\]
and by $q_\alpha$ the bijection
\[
supp(u(\alpha))\to supp(\alpha),\hbox{\ } j_1\mapsto i_2,\ldots,
j_{m-1}\mapsto i_m, j_m\mapsto i_1,
\]
in case $m\geq 2$, and $p_\alpha=q_\alpha=(i_1,j_1)$ in case
$m=1$;

(E$'$) choose a cycle $o$ of the permutation $u$, and set
\begin{equation*}
\varsigma_1^{\left(o\right)}=\coprod_{\alpha\in o}p_\alpha\in S_{X_o}, \hbox{\rm
and\ }
\varsigma_2^{\left(o\right)}=\coprod_{\alpha\in o}q_\alpha\in S_{X_o},
\end{equation*}
where $X_o=\cup_{\alpha\in o}supp(\alpha)$;

(F$'$) set
\begin{equation*}
\varsigma_1=\prod_{o\in\mathcal{C}\left(u\right)}
\varsigma_1^{\left(o\right)}\in S_d, \hbox{\rm and\ }
\varsigma_2=\prod_{o\in\mathcal{C}\left(u\right)}
\varsigma_2^{\left(o\right)}\in S_d.
\end{equation*}

Let $R_d=\cup_{\tau\in S_d}R_{d,\tau}$, and $N_d=\cup_{\tau\in
S_d}N_{d,\tau}$.

The family $(X_o)_{o\in\mathcal{C}\left(u\right)}$ of sets is a
partition of the integer-valued interval $[1,d]$. Let $\theta_o$
be the restriction of the permutation $\theta$ on $X_o\cup X_o+d$.
Let us set $\tau^{\left(o\right)}=\prod_{\alpha\in o}\alpha$, so
$\tau=\prod_{o\in\mathcal{C}\left(u\right)}\tau^{\left(o\right)}$.
Note that each one of the families
\[
(\varsigma_1^{\left(o\right)})_{o\in\mathcal{C}\left(u\right)},\hbox{\
}(\varsigma_2^{\left(o\right)})_{o\in\mathcal{C}\left(u\right)},\hbox{\
}(\theta_o)_{o\in\mathcal{C}\left(u\right)},\hbox{\rm\ and\ }
(\tau^{\left(o\right)})_{o\in\mathcal{C}\left(u\right)},
\]
consists of permutations with pairwise disjoint supports.

Next technical lemma can be proved by inspection.

\begin{lemma}\label{III.2.4} (i) Let $o=(\alpha, u(\alpha), u^2(\alpha),\ldots)$
be a cycle of the permutation $u$, and let us choose cyclic
notations for the members of $o$: $\alpha=(i_1,i_2,\ldots, i_m)$,
$u(\alpha)=(j_1,j_2,\ldots, j_m)$, $u^2(\alpha) =(k_1,k_2,\ldots,
k_m)$,$\ldots$. Then
\[
\theta_o\varsigma_1^{\left(o\right)}\omega^d(\varsigma_2^{\left(o\right)})=
\]
\[
(i_1, j_1+d,i_2, j_2+d,\ldots, i_m, j_m+d)(j_1, k_1+d,j_2, k_2+d,\ldots, j_m, k_m+d)\ldots;
\]
(ii) $\varsigma_1^{\left(o\right)}\varsigma_2^{\left(o\right)}=
\varsigma_2^{\left(o\right)}\varsigma_1^{\left(o\right)}=\tau^{\left(o\right)}$;

(iii) if $\sigma\in N_{d,\tau}$, then the ordered pair
$(\varsigma_1,\varsigma_2)\in R_{d,\tau}$, obtained by the
construction sequence with the same initial segment, is such that
$\sigma =\theta\varsigma_1\omega^d(\varsigma_2)$;

(iv) if $(\varsigma_1,\varsigma_2)\in R_{d,\tau}$, then
$\varsigma_1\varsigma_2=\varsigma_2\varsigma_1=\tau$.

\end{lemma}

The action of a group on itself via conjugation can be used 
in order to prove 

\begin{lemma}\label{III.2.5} Let $W\leq S_d$ be a
permutation group, and let $p(W)$ be the number of conjugacy
classes of the group $W$. Then the number of ordered pairs of elements
of $W$, whose components commute, is $|W|p(W)$.

\end{lemma}

A pair of permutations $\eta$, $\zeta$ from $S_\infty$ is said to
be \emph{braid-like} if $\eta\zeta\neq\zeta\eta$, and
$\eta\zeta\eta=\zeta\eta\zeta$.

\begin{theorem}\label{III.2.12} Let $W\leq S_d$ be a permutation group,
and let $\sigma\in\theta W^{\left(2\right)}$,
$\sigma=\theta\varsigma_1\omega^d(\varsigma_2)$, where
$\varsigma_1,\varsigma_2\in W$. The following three statements are then
equivalent:

(i) the pair of permutations $\sigma$, $\omega^d(\sigma)$ is
braid-like;

(ii) one has $\sigma\in M_W$;

(iii) the permutations $\varsigma_1$ and $\varsigma_2$ commute.

Under these conditions, if $\sigma\in M_{W,\tau}$, $\tau\in W$, then
$\varsigma_1\varsigma_2=\varsigma_2\varsigma_1=\tau$.
If, in addition, $W=S_d$, then parts (i) -- (iii) are equivalent to
each one of

(iv) one has $\sigma\in N_d$;

(v) one has $(\varsigma_1,\varsigma_2)\in R_d$.

\end{theorem}

\begin{proof} Suppose that (i) holds, and let $\eta\in S_d$ be
the permutation with $\sigma(i)=d+\eta(i)$ for any $i\in [1,d]$. We have
\[
\sigma\omega^d(\sigma)\sigma(i)=\sigma\omega^d(\sigma)(d+\eta(i))=d+\sigma\eta(i),
\]
and
\[
\omega^d(\sigma)\sigma\omega^d(\sigma)(i)=\omega^d(\sigma)\sigma(i)=
\omega^d(\sigma)(d+\eta(i))=d+\sigma\eta(i),
\]
for $i\in [1,d]$. Let us set $i(j)=j-d$ where $j\in [d+1,2d]$. We
have
\[
\sigma\omega^d(\sigma)\sigma(j)=\sigma^2(j)=\sigma^2(d+i(j)),
\]
and
\[
\omega^d(\sigma)\sigma\omega^d(\sigma)(j)=
\omega^d(\sigma)\sigma\omega^d(\sigma)(d+i(j))=
\omega^d(\sigma)\sigma(d+\sigma(i(j)))=
\]
\[
\omega^d(\sigma)(d+\sigma(i(j)))=d+\sigma^2(i(j)),
\]
for any $j\in [d+1,2d]$. Therefore the pair of permutations
$\sigma$, $\omega^d(\sigma)$ is braid-like if and only if
$\sigma^2(d+i)=d+\sigma^2(i)$ for any $i\in [1,d]$.
Thus, part (i) is equivalent to part (ii).
A comparison of the equalities~(\ref{III.2.7}),
and $\sigma^2=\varsigma_2\varsigma_1\omega^d(\varsigma_1\varsigma_2)$,
yields the equivalence of parts (ii) and (iii).

Now, suppose $W=S_d$. In accord with Theorem~\ref{III.2.12}, (ii), (iii),
and Lemma~\ref{III.2.5}, we
obtain $|M_d|=p(d)d!$, where $p(d)$ is the number of partitions of
$d$. On the other hand, let for any $\upsilon\in S_d$,
$C(\upsilon)$ be its conjugacy class, and let $C$ be a set of
representatives of the conjugacy classes of $S_d$. We have
\[
|N_d|=\sum_{\tau\in S_d}|N_{d,\tau}|=\sum_{\upsilon\in
C}\sum_{\varsigma\in C\left(\upsilon\right)}|N_{d,\varsigma}|=
\sum_{\upsilon\in
C}\frac{d!}{z_{\varrho\left(\upsilon\right)}}
z_{\varrho\left(\upsilon\right)}=p(d)d!.
\]
Since for any $\tau\in S_d$
we have $N_{d,\tau}\subset M_{d,\tau}$, then $N_d\subset M_d$. Therefore,
$N_{d,\tau}=M_{d,\tau}$ for all $\tau\in S_d$, and, in particular,
parts (ii) and (iv) are equivalent.

Lemma~\ref{III.2.4}, (iii), and (iv), yield that part (iv) implies
part (v), and part (v) implies part (iii), respectively.

\end{proof}

\begin{corollary}\label{III.2.13} If $\sigma\in M_{W,\tau}$
for some $\tau\in W$, and $a\in
\langle\tau\rangle^{\left(2\right)}$,
$a=\tau^k\omega^d(\tau^\ell)$, then $\sigma\in
M_{W,\tau^{k+\ell+1}}$.

\end{corollary}

\begin{proof} Let $\sigma=\theta\varsigma_1\omega^d(\varsigma_2)$, where
$\varsigma_1,\varsigma_2\in W$ commute, and $\tau=\varsigma_1\varsigma_2$.
We have
\[
\sigma a=\theta\varsigma_1\omega^d(\varsigma_2)\tau^k\omega^d(\tau^\ell)=
\theta\varsigma_1\tau^k\omega^d(\varsigma_2\tau^\ell),
\]
and therefore
\[
(\sigma a)^2=\theta\varsigma_1\tau^k\omega^d(\varsigma_2\tau^\ell)
\theta\varsigma_1\tau^k\omega^d(\varsigma_2\tau^\ell)=
c_\theta(\varsigma_1\tau^k\omega^d(\varsigma_2\tau^\ell))
\varsigma_1\tau^k\omega^d(\varsigma_2\tau^\ell)=
\]
\[
\omega^d(\varsigma_1\tau^k)\varsigma_2\tau^\ell
\varsigma_1\tau^k\omega^d(\varsigma_2\tau^\ell)=
\varsigma_2\tau^\ell\varsigma_1\tau^k
\omega^d(\varsigma_1\tau^k\varsigma_2\tau^\ell)=\tau^{k+\ell+1}
\omega^d(\tau^{k+\ell+1}).
\]

\end{proof}

\begin{remark}\label{III.2.14} The equivalence of parts (iii) and (v)
of Theorem~\ref{III.2.12} gives explicit expressions 
of any two permutations in the symmetric group $S_d$, which commute,
in terms of the cyclic structure of their common product. 

\end{remark}

\section{Braid-like permutation groups}

\label{III.3}

We remind the definition of Artin braid group $B_n$ on $n$ strands
as an abstract group: this is the group generated by $n-1$
generators $\beta_1$, $\beta_2$, $\ldots$, $\beta_{n-1}$, subject
to the following braid relations
\[
\beta_r\beta_s=\beta_s\beta_r,\hbox{\ } |r-s|\geq 2,
\]
\[
\beta_r\beta_s\beta_r=\beta_s\beta_r\beta_s,\hbox{\ } |r-s|=1.
\]

A group is said to be \emph{$n$-braid-like} if it is a homomorphic
image of Artin braid group $B_n$.

Let $\sigma\in \theta (S_d)^{\left(2\right)}$, where
$\theta=\theta_1\in\Sigma_n$. The permutations
\[
\sigma_1=\sigma,\hbox{\ }\sigma_2=\omega^d(\sigma),\ldots,
\sigma_{n-1}=\omega^{\left(n-2\right)d}(\sigma)
\]
are from the wreath product $S_d\wr S_n$; let $B_n(\sigma)$ be the
subgroup of $S_d\wr S_n$, generated by them:
\begin{equation*}
B_n(\sigma)=\langle\sigma_1,\sigma_2,\ldots, \sigma_{n-1}\rangle.
\end{equation*}

Let $W\leq S_d$ be a permutation group. We note that if $\sigma\in
\theta W^{\left(2\right)}$, then $B_n(\sigma)$ is a subgroup of
the wreath product $W\wr S_n$.

From now on let us suppose that the pair of permutations $\sigma$,
$\omega^d(\sigma)$, where $\sigma\in \theta S_d\omega^d(S_d)$, is
braid-like, and let $\sigma^2=\tau \omega^d(\tau)$, $\tau\in S_d$.
Using Lemma~\ref{III.2.4}, (i), and Theorem~\ref{III.2.12}, we
have
$\sigma=\prod_{o\in\mathcal{C}\left(u\right)}\sigma^{\left(o\right)}$,
where $\sigma^{\left(o\right)}=
\theta_o\varsigma_1^{\left(o\right)}\omega^d(\varsigma_2^{\left(o\right)})$.
According to Lemma~\ref{III.2.4}, (ii), and
Theorem~\ref{III.2.12}, (i), (iii), applied for $W=S_{X_o}$, we
obtain that the pair of permutations $\sigma^{\left(o\right)}\in
\theta_o(S_{X_o})^{\left(2\right)}$,
$\omega^d(\sigma^{\left(o\right)})$, is braid-like. We set
\[
\sigma_1^{\left(o\right)}=\sigma^{\left(o\right)},\hbox{\ }
\sigma_2^{\left(o\right)}=\omega^d(\sigma^{\left(o\right)}),
\ldots,
\sigma_{n-1}^{\left(o\right)}=\omega^{\left(n-2\right)d}(\sigma^{\left(o\right)}),
\]
\[
B_n(\sigma^{\left(o\right)})=\langle\sigma_1^{\left(o\right)},
\sigma_2^{\left(o\right)},\ldots,
\sigma_{n-1}^{\left(o\right)}\rangle,
\]
and
\[
Y_o=X_o\cup X_o+d\cup\ldots\cup X_o+(n-1)d,\hbox{\ }
o\in\mathcal{C}\left(u\right).
\]
The family of sets $(Y_o)_{o\in\mathcal{C}\left(u\right)}$ is a
partition of the integer-valued interval $[1,nd]$.

\begin{proposition}\label{III.3.30} For any cycle $o$ of 
the permutation $u$, one has:

(i) the permutations
$\sigma^{\left(o\right)}=\sigma_1^{\left(o\right)}$,
$\sigma_2^{\left(o\right)}$, $\ldots$,
$\sigma_{n-1}^{\left(o\right)}$ (respectively, the permutations
$\sigma=\sigma_1$, $\sigma_2$, $\ldots$,
$\sigma_{n-1}$) satisfy the braid relations;

(ii) the group $B_n(\sigma^{\left(o\right)})$ (respectively,
the group $B_n(\sigma)$) is $n$-braid-like with surjective
homomorphism $B_n\to B_n(\sigma^{\left(o\right)})$,
(respectively, $B_n\to B_n(\sigma)$) defined by the map
$\beta_s\mapsto\sigma_s^{\left(o\right)}, s=1,\ldots, n-1$
(respectively, $\beta_s\mapsto\sigma_s, s=1,\ldots, n-1$);

(iii) the group $B_n(\sigma^{\left(o\right)})$ is transitive on
the set $Y_o$, the family $(Y_o)_{o\in\mathcal{C}\left(u\right)}$
consists of all sets of transitivity for $B_n(\sigma)$, and
$B_n(\sigma)$ is a subdirect product of
$B_n(\sigma^{\left(o\right)})$, $o\in\mathcal{C}\left(u\right)$:
$B_n(\sigma)\leq\prod_{o\in\mathcal{C}\left(u\right)}
B_n(\sigma^{\left(o\right)})$.

\end{proposition}

\begin{proof} It is enough to prove parts (i) and (ii)
for the permutations $\sigma^{\left(o\right)}$.

(i) If $|r-s|\geq 2$, then the supports of the permutations
$\sigma_r^{\left(o\right)}$ and $\sigma_s^{\left(o\right)}$ are
disjoint, so they commute. In case $|r-s|=1$ part (i) holds
because the pair of permutations $\sigma_1^{\left(o\right)}$,
$\sigma_2^{\left(o\right)}$, is braid-like.

(ii) Part (i) and \cite[Lemma 1.2]{[25]} yield immediately part
(ii).

(iii) Straightforward inspection.

\end{proof}

Part (iii) of the above theorem yields immediately

\begin{corollary} \label{III.3.31} The group $B_n(\sigma)$ is transitive if and only if
the permutation $u$ of the set $\mathcal{C}\left(\tau\right)$ is a long cycle.

\end{corollary}

The intersection $BW_n(\sigma)=B_n(\sigma)\cap W^{\left(n\right)}$
is a normal subgroup of $B_n(\sigma)$. In particular, if $\tau\in
W$, and $\langle\tau\rangle$ is the cyclic group generated by
$\tau$, then $A_n(\sigma)=B_n(\sigma)\cap
\langle\tau\rangle^{\left(n\right)}$ is an abelian normal subgroup
of both $B_n(\sigma)$ and $BW_n(\sigma)$.

\begin{lemma}\label{III.3.3} One has $\sigma_s^2\in A_n(\sigma)$,
the restriction of the conjugations $c_{\sigma_s}$ on
$A_n(\sigma)$ are involutions, and the following equalities hold:
\[
c_{\sigma_s}(\omega^{sd}(\tau^2))=\omega^{\left(s-1\right)d}(\tau^2),
\]
for $s=1,\ldots, n-1$, and
\[
c_{\sigma_{r+1}}(\sigma_r^2)=c_{\sigma_r}(\sigma_{r+1}^2)=
\omega^{r-1}(\tau)\omega^{\left(r+1\right)d}(\tau),
\]
for $r=1,\ldots, n-2$.

\end{lemma}

\begin{proof} Let $\sigma_1=\theta\varsigma_1\omega^d(\varsigma_2)$.
Theorem~\ref{III.2.12} yields
$\varsigma_1\varsigma_2=\varsigma_2\varsigma_1=\tau$, and
$\sigma_1^2=\tau\omega^d(\tau)\in A_n(\sigma)$. Since the
group $A_n(\sigma)$ is abelian, the conjugations $c_{\sigma_s^2}$
coincide with the identity of $A_n(\sigma)$, that is,
$c_{\sigma_s}=c_{\sigma_s^{-1}}$ on the group $A_n(\sigma)$. It is
enough to prove the equalities for $s=1$ and $r=1$. We have
\[
c_{\sigma_1}(\omega^d(\tau^2))=c_{\varsigma_2}(\tau^2)\omega^{2d}(1)=\tau^2,
\]
\[
c_{\sigma_2}(\sigma_1^2)=\tau\omega^{2d}(c_{\varsigma_1}(\tau))=
\tau\omega^{2d}(\tau),
\]
\[
c_{\sigma_1}(\sigma_2^2)=c_{\varsigma_2}(\tau)\omega^{2d}(\tau)
=\tau\omega^{2d}(\tau).
\]

\end{proof}

\begin{theorem}\label{III.3.4} Let $W\leq S_d$ be a permutation group.
Let us suppose that $\sigma\in\theta W^{\left(2\right)}$, and that
the pair of permutations $\sigma$, $\omega^d(\sigma)$ is
braid-like. Let $\sigma^2=\tau\omega^d(\tau)$, and let $q$ be the
order of $\tau\in W$.

(i) The map $\theta_s\mapsto\sigma_s\bmod{BW_n(\sigma)}$,
$s=1,\ldots, n-1$, can be extended to an isomorphism $\Sigma_n\to
B_n(\sigma)/BW_n(\sigma)$;

(ii) one has $BW_n(\sigma)=A_n(\sigma)$, and $A_n(\sigma)$ is an
abelian group, isomorphic to
 \begin{equation}
 \underbrace{\hbox{\ccc Z}/(q)\coprod\cdots\coprod
 \hbox{\ccc Z}/(q)}_{\hbox{$n-1$ times}}\coprod \hbox{\ccc
Z}/(q_2)\label{III.3.6}
 \end{equation}
 where $q=q_2\delta$, and $\delta$ is the greatest common divisor of $q$ and $2$;

 (iii) the group $B_n(\sigma)$ is an extension of the symmetric group $S_n$
 by the abelian group~(\ref{III.3.6}); if, in addition, $q$ is odd, then this  
 extension splits.  

\end{theorem}

\begin{proof} (i) According to, for example, \cite[Ch. 4, Theorem 4.1]{[25]},
the symmetric group $\Sigma_n$ has a standard presentation by
generators $\theta_1,\ldots, \theta_{n-1}$, and relations
\[
\theta_r\theta_s=\theta_s\theta_r,\hbox{\ } |r-s|\geq 2,
\]
\[
\theta_r\theta_s\theta_r=\theta_s\theta_r\theta_s,\hbox{\ }
|r-s|=1,
\]
and $\theta_s^2=1$, $s=1,\ldots,n-1$. Now,
Proposition~\ref{III.3.30}, (i), and Lemma~\ref{III.3.3}, yield
that $\sigma_s\bmod{BW_n(\sigma)}$, $s=1,\ldots, n-1$, satisfy the
above relations, so the map
$\theta_s\mapsto\sigma_s\bmod{BW_n(\sigma)}$, $s=1,\ldots, n-1$,
can be extended to a surjective homomorphism $\Sigma_n\to
B_n(\sigma)/BW_n(\sigma)$. Since $\sigma_s$, $s=1,\ldots, n-1$,
$\sigma_r\sigma_{r+1}$, $r=1,\ldots, n-2$, $\sigma_1\sigma_3$, and
$\sigma_1^2$, are $2n-1$ in number pairwise different elements
$\bmod{BW_n(\sigma)}$, then the above homomorphism is an
isomorphism
\begin{equation}
\Sigma_n\simeq B_n(\sigma)/BW_n(\sigma),\label{III.3.5}
\end{equation}
if $n\geq 4$. In case $n=3$ this is true because $\sigma_1$,
$\sigma_2$, and $\sigma_1^2$, are pairwise different
$\bmod{BW_3(\sigma)}$.

(ii) By part (i), the elements of $B_n(\sigma)$, which belong to
$W^{\left(n\right)}$ (that is, the elements of $B_n(\sigma)$,
which are equal to the identity element $\bmod{BW_n(\sigma)}$),
are products of conjugates of $\sigma_s^2$, and of their inverses.
Lemma~\ref{III.3.3} implies that all generators of the group
$BW_n(\sigma)$ are
\[
\sigma_s^2,\hbox{\ } s=1,\ldots, n-1,\hbox{\rm\ and\ }
c_{\sigma_r}(\sigma_{r+1}^2)=
\omega^{r-1}(\tau)\omega^{\left(r+1\right)d}(\tau),\hbox{\ }
r=1,\ldots, n-2.
\]
In particular, $BW_n(\sigma)=A_n(\sigma)$.

Let $\hbox{\ccc Z}^n$ be the free $\hbox{\ccc Z}$-module with
standard basis $e_1=(1,0,\ldots, 0)$, $e_2=(0,1,\ldots, 0)$,
$\ldots$, $e_n=(0,0,\ldots, 1)$. There exists a surjective
homomorphism of abelian groups
\begin{equation*}
\hbox{\ccc Z}^n\to \langle\tau\rangle^{\left(n\right)},
\end{equation*}
\[
(r_1,r_2,\ldots,r_n)\mapsto
\tau^{r_1}\omega^{d}(\tau^{r_2})\ldots\omega^{\left(n-1\right)d}
(\tau^{r_n}),
\]
with kernel $(q)\hbox{\ccc Z}^n$. We set
\[
f_1=e_1+e_2,\hbox{\ }f_2=e_2+e_3,\ldots,f_{n-1}=e_{n-1}+e_n,
\hbox{\ }f_n=e_n.
\]
In terms of the basis $\{f_1,f_2,\ldots, f_{n-1}, f_n\}$, the
above homomorphism can be rewritten in the form
\begin{equation}
\hbox{\ccc Z}^n\to
\langle\tau\rangle^{\left(n\right)},\label{III.3.7}
\end{equation}
\[
(\lambda_1,\lambda_2,\ldots,\lambda_{n-1}, \lambda_n)\mapsto
\tau^{\lambda_1}\omega^d(\tau^{\lambda_1+\lambda_2})\ldots
\omega^{\left(n-1\right)d}(\tau^{\lambda_{n-1}+\lambda_n}).
\]

The inverse image $H_n$ of $A_n(\sigma)$ via~(\ref{III.3.7}) is
the $\hbox{\ccc Z}$-submodule of $\hbox{\ccc Z}^n$, generated by
\[
f_1,f_2,\ldots, f_{n-1},\hbox{\rm\ and\ } g_1=e_1+e_3,\hbox{\ }
g_2=e_2+e_4,\hbox{\ }\ldots,\hbox{\ } g_{n-2}=e_{n-2}+e_n.
\]
Let us set $h_i=2e_i$, $i=1,\ldots, n$. We have
\[
g_1=f_1+f_2-h_2,\hbox{\ }g_2=f_2+f_3-h_3,\ldots,
g_{n-2}=f_{n-2}+f_{n-1}-h_{n-1},
\]
and
\[
h_1+h_2=2f_1,\hbox{\ }h_2+h_3=2f_2,\ldots, h_{n-1}+h_n=2f_{n-1}.
\]
 Moreover, the set $\{f_1,f_2,\ldots,
f_{n-1}, h_n\}$, is linearly independent over $\hbox{\ccc Z}$.
Therefore $H_n$ is a $\hbox{\ccc Z}$-submodule of $\hbox{\ccc
Z}^n$ with basis $\{f_1,f_2,\ldots, f_{n-1}, h_n\}$. The
restriction of the homomorphism~(\ref{III.3.7}) on $H_n$ has
kernel $H_n\cap (q)\hbox{\ccc Z}^n$, so we obtain an isomorphism
of $\hbox{\ccc Z}/(q)$-modules
\[
\mu\colon\hbox{\ccc Z}/(q)\coprod\cdots\coprod \hbox{\ccc Z}/(q)
\coprod \hbox{\ccc Z}/(q_2)\to A_n(\sigma),
\]
\[
(\lambda_1\bmod{q},\lambda_2\bmod{q},\ldots,\lambda_{n-1}\bmod{q},
\lambda_n\bmod{q_2})\mapsto
\]
\[
\tau^{\lambda_1}\omega^d(\tau^{\lambda_1+\lambda_2})\ldots
\omega^{\left(n-1\right)d}(\tau^{\lambda_{n-1}+2\lambda_n}).
\]

(iii) We identify the multiplicatively written group $A_n(\sigma)$
with its additively written  version via $\mu$. We also identify
the factorgroup $B_n(\sigma)/A_n(\sigma)$ with the symmetric group
$\Sigma_n$ via the isomorphism~(\ref{III.3.5}), and let
\[
\pi\colon B_n(\sigma)\to \Sigma_n
\]
be the corresponding canonical surjective homomorphism. We have
the short exact sequence of groups
\begin{equation*}
\begin{CD}
0@> >>\hbox{\ccc Z}/(q)\coprod\cdots\coprod \hbox{\ccc Z}/(q)
\coprod \hbox{\ccc Z}/(q_2)
@>\mu>> B_n(\sigma)@>\pi>> \Sigma_n @> >>1.\\
\end{CD}
\end{equation*}
In particular, the group $B_n(\sigma)$ is an extension of the
symmetric group $S_n$ by the abelian group~(\ref{III.3.6}). 

In accord with Theorem~\ref{III.2.12}, (i), (ii), its
Corollary~\ref{III.2.13}, and Proposition~\ref{III.3.30}, (i), if
$a\in \langle\tau\rangle^{\left(2\right)}$,
$a=\tau^k\omega^d(\tau^\ell)$, then the permutations
$\eta_1=\sigma a$, $\eta_2=\omega^d(\eta_1)$, $\ldots$,
$\eta_{n-1}=\omega^{\left(n-2\right)}(\eta_1)$, satisfy the braid
relations, hence the group $B_n(\sigma a)$ is
$n$-braid-like. If, $k+\ell\equiv -1\bmod{q}$, then
$\eta_s^2=1$, $s=1,\ldots, n-1$, and in this case the map
$\theta_s\mapsto\eta_s$, $s=1,\ldots, n-1$, can be extended to a
homomorphism $\rho\colon\Sigma_n\to B_n(\sigma a)$. 
Let us suppose, in addition, that $q$ is odd. 
Then for any two integers $k$, $\ell$, we have $a\in A_n(\sigma)$, 
hence $B_n(\sigma a)$ is a subgroup of 
$B_n(\sigma)$, and the homomorphism $\rho$ splits $\pi$.

\end{proof}

As an immediate consequence of the above theorem, we obtain

\begin{corollary}\label{III.3.10} If the order
$q$ of the permutation $\tau$ is odd, then the abstract group $B_n(\sigma)$
does not depend on the permutation $\sigma$, but only on $q$.

\end{corollary}

We set $\iota_s=c_{\sigma_s}$, $s=1,\ldots, n-1$. Then $\iota_s$
are automorphisms of the group $\hbox{\ccc Z}/(q)\cdots\coprod
\hbox{\ccc Z}/(q) \coprod \hbox{\ccc Z}/(q_2)$. Taking into
account Lemma~\ref{III.3.3}, we have that $\iota_s$ are
involutions with $\iota_s(f_r)=g_{\min\{r,s\}}$ if $|r-s|=1$,
$\iota_s(f_s)=f_s$, $\iota_s(f_r)=f_r$ if $|r-s|\geq 2$, and
$\iota_{n-1}(h_n)=h_{n-1}$, for any $s=1,\ldots, n-1$, and
$r=1,\ldots, n-2$.

\begin{proposition}\label{III.3.11} The monodromy homomorphism
\[
m\colon \Sigma_n\to Aut(\hbox{\ccc Z}/(q)\coprod\cdots\coprod
\hbox{\ccc Z}/(q) \coprod \hbox{\ccc Z}/(q_2)),\hbox{\
}\theta_s\mapsto \iota_s,
\]
that corresponds to the extension from
Theorem~\ref{III.3.4}, (iii), is injective.

\end{proposition}

\begin{proof} Since the automorphisms $\iota_s$, $s=1,\ldots, n-1$,
$\iota_r\iota_{r+1}$, $r=1,\ldots, n-2$, $\iota_1^2=id$, and
$\iota_1\iota_3$, are pairwise different, the homomorphism $m$ is
injective when $n\geq 4$. The existence of three pairwise
involutions $\iota_1$, $\iota_2$, and $\iota_1^2$ yields that the
homomorphism $m$ is injective also for $n=3$.

\end{proof}

\end{document}